\documentclass[11pt,centertags]{amsart}

\usepackage{amscd}
\usepackage{easybmat}
\usepackage{mathrsfs}
\usepackage{amsfonts}
\usepackage{color}
\usepackage{pifont}
\usepackage{upgreek}
\usepackage{bm}
\usepackage{hyperref}
\usepackage{shorttoc}
\usepackage{amsmath,amstext,amsthm,a4,amssymb,amscd}
\usepackage[mathscr]{eucal}
\usepackage{mathrsfs}
\usepackage{epsf}
\numberwithin{equation}{section}

\def\p{\partial}
\def\o{\overline}
\def\b{\bar}
\def\mb{\mathbb}
\def\mc{\mathcal}

\newtheorem{thm}{Theorem}[section]
\newtheorem{lemma}[thm]{Lemma}
\newtheorem{prop}[thm]{Proposition}
\newtheorem{cor}[thm]{Corollary}
\theoremstyle{definition}
\newtheorem{rem}[thm]{Remark}

\newtheorem{problem}[thm]{Problem}

\theoremstyle{definition}
\newtheorem{defn}[thm]{Definition}
\newcommand{\comment}[1]{}

\usepackage{fancyhdr}

\newenvironment{aligns}{\equation\aligned}{\endaligned\endequation}

\begin{document}

\title{Complex Finsler vector bundles with positive Kobayashi curvature}

\author[Huitao Feng]{Huitao Feng$^1$}
\author[Kefeng Liu]{Kefeng Liu$^2$}
\author[Xueyuan Wan]{Xueyuan Wan}

\address{Huitao Feng: Chern Institute of Mathematics \& LPMC,
Nankai University, Tianjin, China}
\email{fht@nankai.edu.cn}

\address{Kefeng Liu:
Mathematical Sciences Research Center, Chongqing University of Science and Technology, Chongqing 400054, China; Department of Mathematics, University of California at Los Angeles, California 90095, USA}
\email{liu@math.ucla.edu}

\address{Xueyuan Wan: School of Mathematics, Korea Institute for Advanced Study, Seoul 02455, Republic of Korea}
\email{xwan@kias.re.kr}

\thanks{$^1$~Partially supported by NSFC(Grant No. 11221091, 11271062, 11571184, 11771070), the Fundamental Research Funds for the Central Universities and Nankai Zhide Foundation.}

\thanks{$^2$~Partially supported by NSF (Grant No. 1510216).}

\begin{abstract} In this short note, we prove that a complex Finsler vector bundle with positive Kobayashi curvature must be ample, which partially solves  a problem of S. Kobayashi posed in 1975. As applications, a strongly pseudoconvex complex Finsler manifold
 with positive Kobayashi curvature must be biholomorphic to the complex projective
space; we also show that all Schur polynomials are numerically positive for complex Finsler vector bundles with positive Kobayashi curvature.

 \end{abstract}
\maketitle

\section*{Introduction}

 Let $\pi:E\to M$ be a holomorphic vector bundle over compact complex manifold $M$. In this paper, we always assume that $\text{rank}E=r$, $\dim M=n$. It is well-known that $E$ is ample in the sense of Hartshorne if and only if the hyperplane line bundle $\mc{O}_{P(E^*)}(1)$ is a positive line bundle over $$P(E^*)=(E^*\setminus\{0\})/\mb{C}^*$$ (see \cite[Proposition 3.2]{Hart}), i.e.  $\mc{O}_{P(E^*)}(1)$ admits a positive curvature metric.
Understanding the relations between the algebraic positivity and the geometric positivity is an important problem. When $E$ itself admits a Hermitian metric of Griffiths positive curvature (see e.g.  \cite[Definition 2.1]{Liu1}), then $E$ is an ample vector bundle. In \cite{Gri}, Griffiths conjectured that its converse also holds, namely there exists a Hermitian metric of Griffiths positive curvature on $E$ if $E$ is ample.
If $M$ is a curve, H. Umemura \cite{Umemura} and  Campana-Flenner \cite{CF} gave an affirmative answer to this conjecture. For the general case, B. Berndtsson \cite{Bern2} proved that $E\otimes \det E$ is Nakano positive, C. Mourougane and S. Takayama \cite{MT} proved that $S^kE\otimes \det E$ is Griffiths positive for any $k>0$.

In 1975, S. Kobayashi \cite{Ko1} obtained the following equivalent description on the ampleness in Finsler setting. The related notations will be introduced in Section \ref{sec1} in this paper. More precisely,
\begin{thm}[{Kobayashi \cite[Theorem 5.1]{Ko1}}]
	$E$ is ample if and only if there exists a strongly pseudoconvex complex Finsler metric on $E^*$ with negative Kobayashi curvature.
\end{thm}
Furthermore, in \cite[Section 5, Page 162]{Ko1}, S. Kobayashi posed the following problem:
\begin{problem}
It is reasonable to expect that $E$ is ample if and only if it admits a complex Finsler structure of positive curvature. The question is whether $E$ admits a complex Finsler structure of positive curvature if and only if $E^*$ admits a complex Finsler structure of negative curvature.
\end{problem}
In this paper, we partially solve this problem affirmatively and obtain
\begin{thm}\label{thm0}
	Let $\pi:E\to M$ be a holomorphic vector bundle over the compact complex manifold $M$. If $E$ admits a strongly pseudoconvex complex Finsler metric with positive Kobayashi curvature, then $E$ is ample.
\end{thm}

It is easy to see that $E$ admitting a Hermitian metric of Griffiths positive curvature is equivalent to the existence of a Griffiths negative Hermitian metric on the dual bundle $E^*$. However, in the Finslerian case, it is very difficult to find such a simple duality as in the Hermitian situation. A natural and direct way suggested by S. Kobayashi \cite[Section 5, page 162]{Ko1} is: for a given complex Finsler structure $G$ in $E$, considering the complex Finsler structure $G^*$ on $E^*$ defined by
\begin{align*}
	G^*(z,\zeta^*)=\sup_{G(z,\zeta)=1}|\langle\zeta^*,\zeta\rangle|^2,
\end{align*}
and trying to check that $G$ has positive curvature if and only if $G^*$ has negative curvature. Apparently, this is very hard to be achieved due to the difficulty in finding more computable relationships between $G$ and $G^*$. On the other hand, we know that if the Finsler metric $G$ is (fiberwise) strictly convex and has positive (resp. negative) curvature, then the dual metric $G^*$ has negative (resp. positive) curvature (see \cite[Theorem 2.5]{Dem} or \cite{Sommese}). However, it is unknown whether the ampleness of $E$ guarantees the existence of a convex strictly plurisubharmonic Finsler metric on $E^*$.

In the following we first briefly introduce our approach to Theorem \ref{thm0}.

For a strongly pseudoconvex complex Finsler metric $G$ on $E$  (cf. Definition \ref{defn1}), we still denote by $G$ the induced metric on the tautological line bundle $\mc{O}_{P(E)}(-1)$. Then the  $(1,1)$-form $\sqrt{-1}\p\b{\p}\log G$ admits a decomposition (cf. \cite{Ko1, FLW}, also  Proposition \ref{prop2}):
 \begin{align*}
\sqrt{-1}\p\b{\p}\log G=-\Psi+\omega_{FS},
 \end{align*}
 where $\omega_{FS}$ is a vertical $(1,1)$-form and positive definite along each fiber of $p:P(E)\to M$, and $\Psi$, the {\it Kobayashi curvature} of the Finsler metric $G$ named in \cite{FLW}, is a horizontal $(1,1)$-form.
The Finsler metric $G$ is of {\it positive (negative) Kobayashi curvature} if $\Psi>0 (<0)$ along horizontal directions (cf. Definition \ref{defn2}).

If $E$ admits a strongly pseudoconvex complex Finsler metric with positive Kobayashi curvature, we prove that $P(E^*)$ is projective (cf. Lemma \ref{lemma2}). In order to prove $E$ is ample or $\mc{O}_{P(E^*)}(1)$ is positive, from our Lemma \ref{lemma3},
it suffices to show for any holomorphic line bundle $F$ on $P(E^*)$, there exists a positive integer $m_0$ such that
	\begin{align}\label{0.2}
	H^i(P(E^*), \mc{O}_{P(E^*)}(m)\otimes F)=0,\quad i>0,\quad m\geq m_0.
	\end{align}
Let $p: P(E)\to M$ and  $p_1: P(E^*)\to M$ denote the natural projections. Since the Picard group of $P(E^*)$ has the following simple structure
\begin{align*}
	\text{Pic}(P(E^*))\simeq\text{Pic}(M)\oplus \mb{Z}\mc{O}_{P(E^*)}(1),
\end{align*}
there exist a line bundle $F_1$ on $M$ and an integer $a\in\mb{Z}$ such that
$F= p_1^*F_1\otimes \mc{O}_{P(E^*)}(a)$. Now by the Serre duality and \cite[Theorem 5.1]{BPV}, for any integer $m\geq -a$, we get (cf.   Proposition \ref{prop5}) for the proof of the following isomorphisms)
\begin{align*}
	 H^i(P(E^*), \mc{O}_{P(E^*)}(m)\otimes F)& \cong H^i(M, S^{m+a}E\otimes F_1)\cong H^{n-i}(M, S^{m+a}E^*\otimes F_1^*\otimes K_M)\\
	&\cong H^{n-i}(P(E), \mc{O}_{P(E)}(m+a)\otimes p^*(F_1^*\otimes K_M)).
\end{align*}
As pointed by Demailly \cite[Section 5.9, Page 247]{Dem}, for any locally free sheaf $\mc{F}$, it holds that
\begin{align*}
	H^q(P(E), \mc{O}_{P(E)}(m)\otimes \mc{F})=0,\quad q\neq n,\quad m\geq m_0	
\end{align*}
 for some integer $m_0>0$. By taking  $\mc{F}=p^*(F_1^*\otimes K_M)$, we finally get (\ref{0.2}) and therefore the ampleness of $E$.

Now we give some applications on Theorem \ref{thm0}. The  following two direct corollaries  follow from the famous theorems of S. Mori \cite[Theorem 8]{Mori}, W. Fulton and R. Lazarsfeld \cite[Theorem I]{FL}:.
\begin{cor}\label{cor0.1}
If $(M,G)$ is a strongly pseudoconvex complex Finsler manifold with positive Kobayashi curvature, then $M$ is biholomorphic to $\mb{P}^n$.
\end{cor}
Note that when the above Finsler metric $G$ is induced from a K\"ahler metric on $M$, Y.-T. Siu and S.-T Yau in \cite{SY} proved this result in a direct geometric way.
\begin{cor}\label{cor0.2}
All Schur polynomials are numerically positive for complex Finsler bundles with positive Kobayashi curvature.
\end{cor}

\vskip 3mm

This article is organized as follows. In Section \ref{sec1}, we shall fix the notations and recall some basic definitions and facts on complex Finsler vector bundles, positive (negative) Kobayashi curvature. In Section \ref{sec2}, we will prove our main Theorem \ref{thm0}. In Section \ref{sec3}, we will give two applications on Theorem \ref{thm0} and prove Corollary \ref{cor0.1}, \ref{cor0.2}.

\vspace{5mm}
{\bf Acknowledgements.} The authors would like to thank Professor  Xiaokui Yang for many helpful discussions.


\section{Complex Finsler vector bundle}\label{sec1}

In this section, we shall fix the notations and recall some basic definitions and facts on complex Finsler vector bundles.  For more details we refer to \cite{Cao-Wong, Dem, FLW, Ko1, Wan1}.

We will use $z=(z^1,\cdots, z^n)$ to denote a local holomorphic coordinate system on $M$ and use $\{e_i\}_{1\leq i\leq r}$ to denote a local holomorphic frame of $E$. Then any element $v$ in $E$ can be written as
$$v=v^i e_i\in E,$$
where we adopt the summation convention of Einstein. In this way, one gets a local holomorphic coordinate system of the complex manifold $E$:
\begin{align}\label{coor}
(z;v)=(z^1,\cdots,z^n; v^1,\cdots, v^r).
\end{align}

\begin{defn}[\cite{Ko1}]\label{defn1}
A Finsler metric  $G$ on the holomorphic vector bundle $E$ is a continuous function $G:E\to\mathbb{R}$ satisfying the following conditions:
\begin{description}
  \item [F1)] $G$ is smooth on $E^o=E\setminus \{0\}$;
  \item[F2)] $G(z,v)\geq 0$ for all $(z,v)\in E$ with $z\in M$ and $v\in\pi^{-1}(z)$, and $G(z,v)=0$ if and only if $v=0$;
  \item[F3)] $G(z,\lambda v)=|\lambda|^2G(z,v)$ for all $\lambda\in\mathbb{C}$.
\end{description}
Moreover, $G$ is called strongly pseudoconvex if it satisfies
\begin{description}
  \item[F4)] the Levi form ${\sqrt{-1}}\partial\bar\partial G$ on $E^o$ is positive definite along each fiber $E_z=\pi^{-1}(z)$ for $z\in M$.
\end{description}
\end{defn}
Clearly, any Hermitian metric on $E$ is naturally a strongly pseudoconvex complex Finsler metric on it.

We write
\begin{align*}
\begin{split}
& G_{i}=\partial G/\partial v^{i},\quad G_{\bar j}=\partial G/\partial\bar{v}^{j},\quad G_{i\bar{j}}=\partial^{2}G/\partial v^{i}\partial\bar{v}^{j},\\
& G_{i\alpha}=\partial^{2}G/\partial v^{i}\partial z^{\alpha}, \quad G_{i\bar j\bar\beta}=\partial^3G/\partial v^{i}\partial\bar v^j\partial\bar z^{\beta},\quad etc.,
\end{split}
\end{align*}
to denote the differentiation with respect to $v^i,\bar v^j$ ($1\leq i,j\leq r$), $z^\alpha,\bar z^\beta$ ($1\leq\alpha,\beta\leq n$).

If $G$ is a strongly pseudoconvex complex Finsler metric on $M$, then there is a canonical h-v decomposition of the holomorphic tangent bundle $TE^o$ of $E^o$ (see \cite[\S 5]{Cao-Wong} or \cite[\S 1]{FLW}).
\begin{align*}
TE^o=\mc{H}\oplus \mc{V}.
\end{align*}
In terms of local coordinates,
\begin{align}\label{HV}
\mc{H}=\text{span}_{\mb{C}}\left\{\frac{\delta}{\delta z^\alpha}=\frac{\p}{\p z^\alpha}-G_{\alpha\b{j}}G^{\b{j}k}\frac{\p}{\p v^k}, 1\leq \alpha\leq n\right\},\quad \mc{V}=\text{span}_{\mb{C}}\left\{\frac{\p}{\p v^i}, 1\leq i\leq r\right\}.
\end{align}
The dual bundle $T^*E^o$ also has a smooth h-v decomposition $T^*E^o=\mc{H}^*\oplus\mc{V}^*$:
\begin{align}
\mc{H}^*=\text{span}_{\mb{C}}\{dz^{\alpha}, 1\leq \alpha\leq n\},\quad \mc{V}^*=\text{span}_{\mb{C}}\{\delta v^i=dv^i+G^{\b{j}i}G_{\alpha\b{j}}dz^\alpha,\quad 1\leq i\leq r\}.
\end{align}

With respect to the h-v decomposition (\ref{HV}), the $(1,1)$-form $\sqrt{-1}\p\b{\p}\log G$ has the following decomposition.
\begin{prop}[\cite{Aikou, Ko1}]\label{prop1}
Let $G$ be a strongly pseudoconvex complex Finsler metric on $E$, one has
\begin{align*}
\sqrt{-1}\p\b{\p}\log G=-\Psi+\omega_{V}
\end{align*}
on $E^o$,
where $\Psi$ and $\omega_V$ are given by
\begin{align}\label{Psi}
\Psi=\sqrt{-1}R_{i\b{j}\alpha\b{\beta}}\frac{v^i\b{v}^j}{G}dz^\alpha\wedge d\b{z}^\beta,\quad \omega_{V}=\sqrt{-1}\frac{\p^2 \log G}{\p v^i\p\b{v}^j}\delta v^i\wedge \delta\b{v}^j,
\end{align}
with
\begin{align*}
	R_{i\b{j}\alpha\b{\beta}}=-\frac{\p^2 G_{i\b{j}}}{\p z^\alpha\p\b{z}^\beta}+G^{\b{l}k}\frac{\p G_{i\b{l}}}{\p z^\alpha}\frac{\p G_{k\b{j}}}{\p\b{z}^\beta}.
\end{align*}
\end{prop}
\begin{defn}[{\cite[Definition 1.2]{FLW}}]\label{defn2}
The form $\Psi$ defined by (\ref{Psi}) is called the {\it Kobayashi curvature} of the complex Finsler vector bundle $(E, G)$. A strongly pseudoconvex complex Finsler metric $G$ is said to be of {\it positive (respectively, negative) Kobayashi curvature} if
\begin{align*}
\left(R_{i\b{j}\alpha\b{\beta}}\frac{v^i\b{v}^j}{G}\right)
\end{align*}
is a positive (respectively, negative) definite matrix on $E^o$.
 \end{defn}
 \begin{rem}
 Note  that  the positive (resp. negative) Kobayashi curvature is a natural generalization of Griffiths positive (resp. negative)  of a Hermitian vector bundle (cf. \cite[Definition 2.1]{Liu1}). In fact, if a Finsler metric comes from a Hermitian metric, then the Finsler metric has positive (resp. negative) Kobayashi curvature is equivalent to  Griffiths positive (resp. negative).
 \end{rem}
 Let $q$ denote the natural projection
\begin{align}
q: E^o\to P(E):=E^o/\mb{C}^*\quad (z; v)\mapsto (z;[v]):=(z^1,\cdots, z^n; [v^1,\cdots, v^r]),	
\end{align}
which gives a local coordinate system of $P(E)$ by
\begin{align}\label{1.2}
(z;w):=(z^1,\cdots, z^n; w^1,\cdots, w^{r-1})=\left(z^1,\cdots,z^n; \frac{v^1}{v^k},\cdots,\frac{v^{k-1}}{v^k},\frac{v^{k+1}}{v^{k}},\cdots, \frac{v^r}{v^k}\right)
\end{align}
on $$U_k:=\{(z,[v])\in P(E), v^k\neq 0\}.$$
Denote by $((\log G)^{\bar{b}a})_{1\leq a,b\leq r-1}$  the inverse of the matrix  $\left((\log G)_{a\b{b}}:=\frac{\p^2\log G}{\p w^a\b{w}^b}\right)_{1\leq a,b\leq r-1}$ and set
\begin{align}\label{1.3}
\delta w^a=d w^a+(\log G)_{\alpha\b{b}}(\log G)^{\b{b}a}dz^{\alpha}.	
\end{align}
By using the coframe $\{\delta w^a\}$ of $\mc{V}^*$, there is a well-defined  vertical $(1,1)$-form on $P(E)$ by
\begin{align}\label{vertical form}
\omega_{FS}:=\sqrt{-1}\frac{\p^2\log G}{\p w^a\p\b{w}^b}\delta w^a\wedge \delta\b{w}^b.	
\end{align}
From \cite[Lemma 1.4, Remark 1.5, Proposition 1.6]{Wan1} one has
\begin{prop}\label{prop2}
\begin{itemize}
\item[(i)] $q^*\omega_{FS}=\omega_V$;
\item[(ii)] $\sqrt{-1}\p\b{\p}\log G=-\Psi+\omega_{FS}$ on $P(E)$;
\item[(iii)] A Finsler metric $G$ is strongly pseudoconvex if and only if $\omega_{FS}$ is positive definite along each fiber of $p:P(E)\to M$.
\end{itemize}
\end{prop}
\begin{proof}
 (i) is exactly \cite[Lemma 1.4]{Wan1} and (ii) follows directly from (i). The proof of (iii) can be found in \cite[Proposition 1.6]{Wan1}. For readers' convenience, we give the proof here.

  By Definition \ref{defn1}, $G$ is strongly pseudoconvex if $(G_{i\b{j}})$ is a positive definite matrix, which gives a Hermitian metric $\langle\cdot,\cdot\rangle$ on the vertical subbundle $\mc{V}$.
	Denote $v=v^i\frac{\p}{\p v^i}$. If $G$ is strongly pseudoconvex, then for any $u=u^i\frac{\p}{\p v^i}\in\mc{V}$,
	\begin{align*}
	\begin{split}
	(-\sqrt{-1})\omega_{V}(u,\b{u}) &=\frac{1}{G^2}(GG_{i\b{j}}-G_iG_{\b{j}})u^i\b{u}^j\\
	&=\frac{1}{G^2}(\|u\|^2\|v\|^2-|\langle u,v\rangle|^2)\geq 0,	
	\end{split}
	\end{align*}
the equality holds if and only if $u=\lambda v$ for some constant $\lambda\in\mb{C}$. So $\omega_V$ has $r-1$ positive eigenvalues and one zero eigenvalue. Since $\omega_V(v,\o{v})=0$ and $q_*(v)=0$ and by (i), $\omega_{FS}$ is positive definite along each fiber of $p: P(E)\to M$.
Conversely, if $\omega_{FS}$ is positive definite along each fiber, then $\omega_V=q^*\omega_{FS}$ has $r-1$ positive eigenvalues and one zero eigenvalue. So $\omega_V(v,\o{v})=0$ and
\begin{align*}
G_{i\b{j}}u^i\b{u}^j=\frac{1}{G}|G_i u^i|^2+G(-\sqrt{-1})\omega_V(u,\o{u})\geq 0.	
\end{align*}
 Moreover, $G_{i\b{j}}u^i\b{u}^j=0$ if and only if $u=\lambda v$ and $G_iv^i=0$, if and only if $\lambda=0$. So $(G_{i\b{j}})$ is a positive definite matrix.
\end{proof}

\section{Positive Kobayashi curvature} \label{sec2}

In this section, we will prove a complex Finsler vector bundle  with positive Kobayashi curvature must be ample.

Let $G$ be a strongly pseudoconvex complex Finsler metric on $E$ with positive Kobayashi curvature, that is,
\begin{align}
\sqrt{-1}\p\b{\p}\log G=-\Psi+\omega_{FS}	
\end{align}
  with $\Psi>0$ on horizontal subbundle $\mc{H}$, and $\omega_{FS}>0$ along each fiber of $p:P(E)\to M$. Then
  \begin{lemma}\label{lemma2}
  	If $E$ admits  a strongly pseudoconvex complex Finsler metric with positive Kobayashi curvature, then $P(E^*)$ is projective.
  \end{lemma}
\begin{proof}
From \cite[Lemma 2.3]{FLW}, the first Chern class $c_1(\det E)$ satisfies
\begin{aligns}
c_1(\det E)&=c_1(E)=-\int_{P(E)/M}c_1(\mc{O}_{P(E)}(1))^r\\
&=\left[-\int_{P(E)/M}\left(\frac{\sqrt{-1}}{2\pi}\p\b{\p}\log G\right)^r\right]\\
&=\left[\frac{r}{(2\pi)^r}\int_{P(E)/M}\Psi\wedge\omega^{r-1}_{FS}\right].
\end{aligns}
	By assumption, $\int_{P(E)/M}\Psi\wedge\omega^{r-1}_{FS}$ is a positive $(1,1)$-form on $M$, which yields that $\det E$ is a positive line bundle. By take $k$ large enough, the line bundle $$p_1^*(\det E)^{k}\otimes \mc{O}_{P(E^*)}(1)\to P(E^*)$$ is also a positive line bundle, where $p_1: P(E^*)\to M$. By Kodaira embedding theorem, $P(E^*)$ is projective.
\end{proof}
The following lemma can be found in the proof of \cite[Lemma 5.2]{Shiffman}.
\begin{lemma}\label{lemma3}
Let $L\to M$ be a line bundle over projective manifold $M$, and satisfies for any  line bundle $F$ over $M$, there exists an integer $m_0>0$, $H^i(M,F\otimes L^m)=0$ for $i>0$, $m\geq m_0$. Then $L$ is a  positive line bundle over $M$.
\end{lemma}
\begin{proof}
For any coherent sheaf $\mc{F}$ over the projective algebraic manifold $M$, there is a resolution
$$0\to\mc{K}\to \oplus_{s_n}E_n\to \cdots\to\oplus_{s_2}E_2\to \oplus_{s_1}E_1\to \mc{F}\to 0,$$	
where the $E_i$ are holomorphic line bundles over $M$. By the Hilbert syzygy theorem, $\mc{K}=V$ for some holomorphic vector bundle $V$ on $M$. Hence for $m$ sufficiently large, we obtain
\begin{align}
H^i(M,L^m\otimes \mc{F})=H^{n+i}(M,L^m\otimes V)=0,\quad i>0,\quad m\geq m_0	
\end{align}
for some positive integer $m_0$.
By Cartan-Serre-Grothendieck theorem, see e.g. \cite[Theorem 5.1]{Shiffman}, $L$ is a positive line bundle.
\end{proof}

The following vanishing theorem appeared in \cite[Section 5.9]{Dem}, which can be proved by Andreotti-Grauert theorem \cite[Theorem 14]{AG}.
\begin{lemma}\label{lemma4}
 	If $E$ admits  a strongly pseudoconvex complex Finsler metric with positive Kobayashi curvature, then for any holomorphic line bundle $F$ on $P(E)$ there exists an integer $m_0>0$ such that
\begin{align}
H^q(P(E), \mc{O}_{P(E)}(m)\otimes F)=0,\quad q< n,\quad m\geq m_0.	
\end{align}
\end{lemma}
\begin{proof}
By Serre duality, one has 	
\begin{align}\label{2.4}
H^q(P(E), \mc{O}_{P(E)}(m)\otimes F)\cong	H^{n+r-1-q}(P(E), \mc{O}_{P(E)}(-m)\otimes F^*\otimes K_{P(E)}).
\end{align}
Since the curvature form of $\mc{O}_{P(E)}(-1)$ is
\begin{align}
\sqrt{-1}\b{\p}\p\log G=\Psi-\omega_{FS},	
\end{align}
which has $n$ positive eigenvalues at each point of $P(E)$, so $\mc{O}_{P(E)}(-1)$ is $(r-1)$-positive. By  Andreotti-Grauert theorem \cite[Theorem 14]{AG} (see also \cite[Proposition 2.1]{DPS}), $\mc{O}_{P(E)}(-1)$ is $(r-1)$-ample, that is,  for any coherent sheaf $\mc{F}$ on $P(E)$ there exists a positive integer $m_0$  such that
$$H^i(P(E), \mc{F}\otimes \mc{O}_{P(E)}(-m))=0,\quad i>r-1,\quad m\geq m_0.$$ By taking $\mc{F}=F^*\otimes K_{P(E)}$ one gets
\begin{align}
H^{n+r-1-q}(P(E), \mc{O}_{P(E)}(-m)\otimes F^*\otimes K_{P(E)})=0,\quad q<n,\quad m\geq m_0.	
\end{align}
Combining with (\ref{2.4}) completes the proof.
\end{proof}
Our main result in this section is the following.
\begin{thm}\label{prop5}
	If $E$ admits  a strongly pseudoconvex complex Finsler metric with positive Kobayashi curvature, then $E$ is ample.
\end{thm}
\begin{proof}
	From Lemma \ref{lemma2}, \ref{lemma3}, it suffices  to show for any holomorphic line bundle $F$ on $P(E^*)$, there exists a positive integer $m_0$ such that
	\begin{align}
	H^i(P(E^*), \mc{O}_{P(E^*)}(m)\otimes F)=0,\quad i>0,\quad m\geq m_0.	
	\end{align}
Denote $p:P(E)\to M$ and $p_1:P(E^*)\to M$. Since the Picard group of $P(E^*)$  has the following simple structure,
\begin{align}
	\text{Pic}(P(E^*))\simeq\text{Pic}(M)\oplus \mb{Z}\mc{O}_{P(E^*)}(1),
\end{align}
so there exist a line bundle $F_1$ on $M$ and an integer $a\in\mb{Z}$ such that
\begin{align}
	F= p_1^*F_1\otimes \mc{O}_{P(E^*)}(a).
\end{align}
It follows that
\begin{align}\label{3.3}
 	\mc{O}_{P(E^*)}(m)\otimes F=\mc{O}_{P(E^*)}(m+a)\otimes p_1^*F_1.
 \end{align}
For any integer $m\geq -a$, one has
\begin{aligns}\label{3.2}
	&\quad  H^i(P(E^*), \mc{O}_{P(E^*)}(m)\otimes F)\\
	&\cong H^i(P(E^*), \mc{O}_{P(E^*)}(m+a)\otimes p_1^*F_1)\quad \text{by (\ref{3.3})}\\
	&\cong H^i(M, S^{m+a}E\otimes F_1)\quad \text{by \cite[Theorem 5.1]{BPV}}\\
	&\cong H^{n-i}(M, S^{m+a}E^*\otimes F_1^*\otimes K_M) \quad \text{by Serre duality, see e.g. \cite[Corollary 2.11]{Shiffman}}\\
	&\cong H^{n-i}(P(E), \mc{O}_{P(E)}(m+a)\otimes p^*(F_1^*\otimes K_M)) \quad \text{by \cite[Theorem 5.1]{BPV}}\\
	&=0,\quad i>0,\quad m\geq m_0\quad \text{by Lemma \ref{lemma4}}.
\end{aligns}
Thus, $E$ is ample.
\end{proof}

\section{Applications}\label{sec3}

In this section, we will give some applications of Theorem \ref{thm0}. From Theorem \ref{thm0}, it is possible that many related results under the assumption of ampleness could be valid by changing the assumption of ampleness into the existence of a strongly pseudoconvex  complex Finsler metric with positive Kobayashi curvature. 
Among them we only mention here
the following two famous theorems due to S. Mori \cite[Theorem 8]{Mori}, W. Fulton and R. Lazarsfeld \cite[Theorem I]{FL}. 

In \cite{Mori}, Mori proved the following theorem, which solves Hartshorne's conjecture.
\begin{thm}[{\cite[Theorem 8]{Mori}}]
Every irreducible $n$-dimensional non-singular projective variety with ample tangent bundle defined over an algebraically closed field $k$ of characteristic $\geq 0$ is isomorphic to the projective space $\mb{P}^n_k$.	
\end{thm}

In the case $k=\mb{C}$, by Theorem \ref{thm0}, we obtain
\begin{cor}\label{cor1}
If $(M,G)$ is a strongly pseudoconvex complex Finsler manifold with positive Kobayashi curvature, then $M$ is biholomorphic to $\mb{P}^n$.
\end{cor}
Note that when the above Finsler metric $G$ is induced from a K\"ahler metric on $M$, Y.-T. Siu and S.-T Yau in \cite{SY} proved this result in a direct geometric way.

Let $P\in  \mb{Q}[c_1, \ldots, c_r]$ be a weighted homogeneous polynomail of degree $n$, the variable $c_i$ being assigned weight $i$. We say that $P$ is numerically positive for ample vector bundles (resp. complex Finsler vector  bundles of positive Kobyashi curvature) if for every projective variety $M$ of dimension $n$, and every ample vector bundle $E$ (resp. complex Finsler  vector bundles $E$ of positive Kobyashi curvature) over $M$, the Chern number
\begin{align}
\int_M P(c_1(E),\ldots, c_r(E))>0,	
\end{align}
where $r=\text{rank}E$. Denote by $\Lambda(n,r)$ the set of all partitions of $n$ by non-negative integers $<r$. Thus an element $\lambda \in  \Lambda(n, r)$ is specified by a sequence
\begin{align}
r\geq \lambda_1\geq \lambda_2\geq\cdots\geq \lambda_n\geq 0\quad \text{with}\quad \sum\lambda_i=n.	
\end{align}
Each $\lambda\in\Lambda(n,r)$ gives rise to a Schur polynomial $P_{\lambda}\in \mb{Q}[c_l, \ldots, c_r]$ of degree $n$, defined as the $n\times n$ determinant
\begin{align}
P_{\lambda}=\begin{vmatrix}
c_{\lambda_1} & c_{\lambda_1+1} &\cdots & c_{\lambda_1+n-1} \\
c_{\lambda_2-1}& c_{\lambda_2} &\cdots & c_{\lambda_2+n-2}\\
\vdots &\vdots &\ddots &\vdots \\
c_{\lambda_n-n+1} & c_{\lambda_n-n+2} & \cdots & c_{\lambda_n}
\end{vmatrix},
\end{align}
where by convention $c_0=1$ and $c_i=0$ if $i\not\in[0,r]$. The Schur polynomials $P_{\lambda} (\lambda\in \Lambda(n,r))$ form a basis for the $\mb{Q}$-vector space of weighted homogeneous polynomials of degree $n$ in $r$ variables. Given such a polynomail $P$, write
\begin{align}
P=\sum_{\lambda\in \Lambda(n,r)}a_\lambda(P)P_{\lambda}	 \quad (a_{\lambda}(P)\in \mb{Q})
\end{align}
In \cite{FL}, W. Fulton and R. Lazarsfeld obtained the following theorem, which generalized the result of S. Bloch and D. Gieseker \cite{BG}.
\begin{thm}[{\cite[Theorem I]{FL}}]
	The polynomial $P$ is numerically positive for ample vector bundles if and only if $P$ is non-zero and
	\begin{align*}
	a_{\lambda}(P)\geq 0 \quad \text{for all}\quad \lambda\in \Lambda(n,r).	
	\end{align*}
\end{thm}
Combining with Theorem \ref{thm0}, we obtain
\begin{cor}\label{Cor2}
All Schur polynomials are numerically positive for complex Finsler bundles with positive Kobayashi curvature.
\end{cor}
 We should stress that how to prove these two famous theorems in a purely differential geometrical method is still widely open.

\end{document}